\pdfoutput=1
\documentclass[paper=a4,english,fontsize=11pt,parskip=half,abstract=true]{scrartcl}
\usepackage{babel}
\usepackage[utf8]{inputenc}
\usepackage[T1]{fontenc}
\usepackage[left=20mm,right=20mm,top=30mm,bottom=30mm]{geometry}
\usepackage{amsmath}
\usepackage{amsthm}
\usepackage{amssymb}
\usepackage{enumerate} 
\usepackage{thmtools}
\usepackage{mathtools}
\mathtoolsset{centercolon} 
\usepackage[bookmarks=true,
            pdftitle={Groups of p-central type},
            pdfauthor={Benjamin Sambale},
            pdfkeywords={},
            pdfstartview={FitH}]{hyperref}

\newtheorem{Thm}{Theorem} 
\newtheorem{Lem}[Thm]{Lemma}
\newtheorem{Prop}[Thm]{Proposition}

\numberwithin{equation}{section}

\allowdisplaybreaks[1]

\renewcommand{\phi}{\varphi}
\newcommand{\C}{\mathrm{C}}
\newcommand{\N}{\mathrm{N}}
\newcommand{\Z}{\mathrm{Z}}
\newcommand{\pcore}{\mathrm{O}}
\newcommand{\E}{\mathrm{E}}
\newcommand{\F}{\mathrm{F}}

\newcommand{\QQ}{\mathbb{Q}}

\newcommand{\NN}{\mathbb{N}}

\newcommand{\Aut}{\operatorname{Aut}}

\newcommand{\Out}{\operatorname{Out}}

\newcommand{\SL}{\operatorname{SL}}
\newcommand{\PGL}{\operatorname{PGL}}
\newcommand{\PSU}{\operatorname{PSU}}
\newcommand{\PSL}{\operatorname{PSL}}

\newcommand{\Irr}{\mathrm{Irr}}
\newcommand{\IBr}{\mathrm{IBr}}

\newcommand{\Ker}{\operatorname{Ker}}

\title{Groups of $p$-central type}
\author{Benjamin Sambale\footnote{Institut für Algebra, Zahlentheorie und Diskrete Mathematik, Leibniz Universität Hannover, Welfengarten 1, 30167 Hannover, Germany,
\href{mailto:sambale@math.uni-hannover.de}{sambale@math.uni-hannover.de}}}
\date{\today}

\begin{document}
\frenchspacing
\maketitle
\renewcommand{\sectionautorefname}{Section}

\begin{abstract}\noindent
A finite group $G$ with center $Z$ is of central type if there exists a fully ramified character $\lambda\in\Irr(Z)$, i.\,e. the induced character $\lambda^G$ is a multiple of an irreducible character. Howlett--Isaacs have shown that $G$ is solvable in this situation. A corresponding theorem for $p$-Brauer characters was proved by Navarro--Späth--Tiep under the assumption that $p\ne 5$. We show that there are no exceptions for $p=5$, i.\,e. every group of $p$-central type is solvable. 
Gagola proved that every solvable group can be embedded in $G/Z$ for some group $G$ of central type. We generalize this to groups of $p$-central type. As an application we construct some interesting non-nilpotent blocks with a unique Brauer character. This is related to a question by Kessar and Linckelmann.
\end{abstract}


\section{Introduction}

For an irreducible character $\chi$ of a finite group $G$ it is easy to show that $\chi(1)^2\le|G:Z|$ where $Z:=\Z(G)$ denotes the center of $G$ (see \cite[Corollary~2.30]{Isaacs}). If equality holds for some $\chi\in\Irr(G)$, we say that $G$ has \emph{central type}. It was conjectured by Iwahori--Matsumoto and eventually proved by Howlett--Isaacs~\cite{HowlettIsaacs}, using the classification of finite simple groups (CFSG), that all groups of central type are solvable (see also \cite[Chapter~8]{Navarro2}). Apart from this there are no restrictions on the structure of groups of central type. In fact, Gagola~\cite[Theorem~1.2]{GagolaFully} has shown that every solvable group can be embedded into $G/Z$ for some group $G$ of central type. 

It is well-known that $G$ is of central type if there exists some \emph{fully ramified} character $\lambda\in\Irr(Z)$, i.\,e. $\lambda^G=e\chi$ for some integer $e$ and some $\chi\in\Irr(G)$ (in this case $\chi(1)^2=e^2=|G:Z|$; see \cite[Lemma~8.2]{Navarro2}). This can be carried over to Brauer characters over an algebraically closed field of characteristic $p>0$. We call $\lambda\in\IBr(Z)$ \emph{fully ramified} if $\lambda^G=e\phi$ for some integer $e$ and $\phi\in\IBr(G)$. In this situation, Navarro--Späth--Tiep~\cite[Theorem~A]{NavarroFully} proved (again relying on the CFSG) that $G$ is solvable unless $p=5$. Using the structure of a minimal counterexample for $p=5$ (as described in \cite{NavarroFully}), we are able to eliminate this exceptional case. 

\begin{Thm}\label{mainNST}
Let $G$ be a finite group, $Z\unlhd G$ and $\lambda\in\IBr(Z)$ be $G$-invariant. Suppose that $\lambda^G=e\phi$ for some integer $e$ and $\phi\in\IBr(G)$. Then $G/Z$ is solvable.
\end{Thm}

Notice that \autoref{mainNST} generalizes the Howlett--Isaacs theorem by choosing a prime which does not divide $|G|$.
Assuming that $G$ is ($p$-)solvable, there exists a fully ramified Brauer character in $Z=\Z(G)$ if and only if $\phi(1)^2=|G:Z|_{p'}$ for some $\phi\in\IBr(G)$ where $n_{p'}$ denotes the $p'$-part of an integer $n$ (see \autoref{proppcentral} below). If this is the case, we call $G$ a group of $p$-\emph{central type}. 
Our second objective is to carry over Gagola's theorem as follows.

\begin{Thm}\label{mainGagola}
Let $G$ be a solvable group and $p$ be a prime. Then there exists a group $H$ of $p$-central type such that $G$ is isomorphic to a subgroup of $H/\Z(H)$. 
\end{Thm}

Further properties of the group $H$ in \autoref{mainGagola} are stated in \autoref{ThmGa} below. Theorems~\ref{mainNST} and \ref{mainGagola} are proved in the next section.

In \autoref{secappl} we apply our results to $p$-blocks $B$ of $G$ with defect group $D$. The omnipresent \emph{nilpotent} blocks introduced by Broué--Puig~\cite{BrouePuig} have a unique irreducible Brauer character, i.\,e. $l(B):=|\IBr(B)|=1$. 
We are interested in non-nilpotent blocks with $l(B)=1$. These occur far less frequent, e.\,g. they cannot be principal blocks. Indeed, it was speculated by Kessar and Linckelmann that all such blocks are Morita equivalent to their Brauer correspondent in $\N_G(D)$ (if $D$ is abelian, this is equivalent to Broué's conjecture by \cite[Theorem~3]{ZimmermannDerived}). Using Külshammer's reduction~\cite{Kpsolv} and \autoref{mainNST}, all such blocks should occur in solvable groups up to Morita equivalence.

Concrete examples can be build as follows: 
If $H$ is a non-trivial $p'$-group of central type acting on an elementary abelian $p$-group $V$ with kernel $\Z(H)$, then $G=V\rtimes H$ has a non-nilpotent $p$-block $B$ with $l(B)=1$ (see proof of \autoref{mainplength} below). 
The smallest instance in terms of $|G|$ is the unique non-principal block of $G\cong C_3^2\rtimes D_8\cong\mathtt{SmallGroup}(72,23)$, which first appeared in Kiyota~\cite{Kiyota}. Kessar~\cite{KessarC3C3} has shown that every non-nilpotent block with $l(B)=1$ and defect group $C_3^2$ is Morita equivalent to this block (see also \cite{KLR}). 
Similar examples with abelian defect group were investigated in \cite{BKL,BensonGreen,HKone,HulB1,LS}. All these blocks arise from groups with a normal abelian Sylow $p$-subgroup. As an application of \autoref{mainGagola}, we illustrate that non-nilpotent blocks with $l(B)=1$ occur in arbitrarily “complicated” solvable groups.

\begin{Thm}\label{mainplength}
Let $p$ be a prime and $l\ge 1$ an integer. Then there exists a solvable group $G$ of $p$-length $l_p(G)\ge l$ such that $G$ has a non-nilpotent $p$-block $B$ of maximal defect with $l(B)=1$.
\end{Thm}

On the other hand, we construct some non-solvable examples using the following recipe.

\begin{Thm}\label{mainnonsolv}
Let $b$ be a $p$-block of a finite group $H$ such that $l(b)$ is not a $p$-power (in particular $l(b)>1$). Suppose that there exists an automorphism group $A\le\Aut(H)$ such that $A$ acts regularly on $\IBr(b)$. Let $Q$ be a $p$-group upon $A$ acts non-trivially. Then the group $G=(H\times Q)\rtimes A$ where $A$ acts diagonally has a non-nilpotent $p$-block $B$ with $l(B)=1$.
\end{Thm}

Finally we study lifts of Brauer characters. By the Fong--Swan theorem, every irreducible Brauer character $\phi$ of a $p$-solvable group $G$ is the restriction of an ordinary character to the set of $p$-regular elements.  
This is no longer true in non-solvable groups. However, if $\IBr(B)=\{\phi\}$, then Malle--Navarro--Späth~\cite{MNS} proved (using the CFSG) that $\phi$ still has such a lift. For nilpotent blocks with defect group $D$, the number of these lifts is $|D:D'|$. In particular, $\phi$ has at least $p^2$ lifts unless $|D|\le p$. For non-nilpotent blocks, it was conjectured by Malle--Navarro~\cite{MNchardeg} (in combination with Olsson's conjecture) that the number of lifts is strictly less than $|D:D'|$. Answering a question of G.~Navarro, we construct (non-nilpotent) blocks with $\IBr(B)=\{\phi\}$ such that $\phi$ has a \emph{unique} lift.

\section{Groups of \texorpdfstring{$p$}{p}-central type}

For a finite group $G$ we denote the Fitting subgroup by $\F(G)$ and the layer by $\E(G)$ (the product of all components of $G$). Then the generalized Fitting subgroup is a central product $\F^*(G)=\F(G)*\E(G)$. Our notation for characters follows Navarro's book~\cite{Navarro}. 

We need the following lemma from the theory of group extensions (see \cite[Theorem~15.3.1]{Hall}).

\begin{Lem}\label{lemext}
Let $N$ be a finite group, $\alpha\in\Aut(N)$ and $m\in\NN$. Then the following assertions are equivalent:
\begin{enumerate}[(1)]
\item There exists a finite group $H$ such that $N\unlhd H$, $H/N=\langle hN\rangle\cong C_m$ and $\alpha(x)=hxh^{-1}$ for all $x\in N$.
\item There exists $n\in N$ such that $\alpha(n)=n$ and $\alpha^m(x)=nxn^{-1}$ for all $x\in N$.
\end{enumerate}
\end{Lem}

\begin{proof}[Proof of \autoref{mainNST}]
By \cite[Theorem~A]{NavarroFully}, we may assume that $p=5$. 
Let $G$ be a minimal counterexample with respect to $|G:Z|$. Then by \cite[Theorem~8.1 and its proof]{NavarroFully}, the following holds:
\begin{enumerate}[(i)]
\item $Z=\Z(G)=\F(G)$ is a cyclic $\{2,3\}$-group.
\item $\E(G)=T_1*\ldots*T_m$ where $T_1\cong\ldots\cong T_m\cong 6.A_6$ and $\Z(T_1)=\ldots=\Z(T_m)\cong C_6$. (In the notation of \cite{NavarroFully} we have $T_i=S_i'$.)
\item $G/\F^*(G)$ is a $2$-group, which permutes the $T_i$ transitively. In particular, $m=2^n$ for some $n\ge 0$.
\item\label{sylct} Every Sylow $2$-subgroup of $G$ is of central type.
\end{enumerate}
Let $N:=\bigcap_{i=1}^m\N_G(T_i)\unlhd G$ and $a\in N$. Suppose by way of contradiction that $a\notin\F^*(G)$. Since $\F^*(G)$ is self-centralizing and $\Out(T_i)\cong C_2^2$, we have $a^2\in\F^*(G)$. Let $t_i\in T_i$ and $z\in Z$ such that $a^2=t_1\ldots t_mz$.
If $a$ induces an inner automorphism on every $T_i$, then we have the contradiction $a\in\F^*(G)\C_G(\F^*(G))\le\F^*(G)$. 
Thus, without loss of generality we may assume that $a$ induces an outer automorphism $\alpha$ of $T_1$. Then $\alpha(t_1)=t_1$ and $\alpha^2$ is the inner automorphism on $T_1$ induced by $t_1$. Hence, by \autoref{lemext}, there exists a group $H$ with $T_1\unlhd H$ and $H/T_1=\langle hT_1\rangle\cong C_2$ where $h$ acts as $\alpha$ on $T_1$. Since $\Z(T_1)=\Z(\E(G))\le\F(G)=\Z(G)$, we have $\Z(Z_1)\le\Z(H)$. Moreover, $H/\Z(T_1)$ is isomorphic to one of the three subgroups of $\Aut(A_6)$ of index $2$, namely $S_6$, $\PGL(2,9)$ or the Mathieu group $M_{10}$. However, none of these three groups has a $6$-fold Schur extension as can be checked with GAP~\cite{GAPnew}. This contradiction shows that $N=\F^*(G)$. 

Now the $2$-group $G/N$ is a transitive permutation group of degree $m$. Therefore, $|G:N|$ is bounded by the order of a Sylow $2$-subgroup of the symmetric group $S_m$. This yields 
\begin{equation}\label{eqw}
w:=|G:N|\le2^{\sum_{i=1}^{n}\frac{m}{2^i}}=2^{1+2+\ldots+2^{n-1}}=2^{m-1}.
\end{equation}
Let $P$ be a Sylow $2$-subgroup of $G$ and $Z_2=P\cap Z\le\Z(P)$. Observe that $\C_P(P\cap N)\le N$. Since $T_1$ has Sylow $2$-subgroups isomorphic to $Q_{16}$, we have $P\cap N\cong Q_{16}*\ldots*Q_{16}*Z_2$ and consequently $\Z(P)\le\C_P(P\cap N)=Z_2$. Thus, $\Z(P)=Z_2$.
Choose $v_i\in T_i\cap P$ of order $8$ for $i=1,\ldots,m$. Since $[T_i,T_j]=1$ for $i\ne j$, $A:=\langle v_1,\ldots,v_m\rangle Z_2\le P$ is abelian with $|P:A|=|G:N||P\cap N:A|=2^mw$.
By part~\eqref{sylct}, there exists $\chi\in\Irr(P)$ such that $\chi(1)^2=|P:Z_2|$. For a (linear) constituent $\mu\in\Irr(A)$ of the restriction $\chi_A$ we have $\chi(1)\le\mu^P(1)=|P:A|$. Hence,
\[2^{3m}w=|G:N||N:Z|_2=|P:Z_2|=\chi(1)^2\le|P:A|^2=2^{2m}w^2\]
and $w\ge 2^m$. This contradicts \eqref{eqw}.
\end{proof}

The original conjecture by Iwahori--Matsumoto has been generalized by J.~F.~Humphreys (as mentioned in \cite{Higgs}) and independently by Navarro~\cite[Conjecture~11.1]{NavarroProblems} to the following statement: Let $N\unlhd G$ and $\lambda\in\Irr(N)$ be $G$-invariant such that all irreducible constituents of $\lambda^G$ have the same degree. Then $G/N$ is solvable. This conjecture is still open, but the corresponding version for Brauer characters does not hold. Indeed the Schur cover $6.A_6$ considered in the proof above is a counterexample for $p=5$. 

Next, we prove the equivalent characterization of groups of $p$-central type mentioned in the introduction.

\begin{Prop}\label{proppcentral}
For a $p$-solvable group $G$ with center $Z:=\Z(G)$ the following assertions are equivalent:
\begin{enumerate}[(1)]
\item There exists some $\phi\in\IBr(G)$ such that $\phi(1)^2=|G:Z|_{p'}$.
\item There exists a fully ramified Brauer character $\lambda\in\IBr(Z)$.
\end{enumerate}
\end{Prop}
\begin{proof}
It has been shown in \cite[Theorem~2.1]{NavarroFully} that (2) implies (1). The other implication was claimed in \cite{NavarroFully} without proof. So assume that (1) holds and choose a constituent $\lambda\in\IBr(Z)$ of $\phi_Z$. Since $G$ is $p$-solvable, there exists a $p$-complement $H\le G$. 
Since $\phi(1)$ has $p'$-degree, the restriction $\phi_H$ is irreducible by \cite[Theorem~10.9]{Navarro}. We may consider $\phi_H$ as an ordinary character. 
Let $Z=Z_p\times Z_{p'}$ and $\lambda=\lambda_p\times\lambda_{p'}$ where $Z_p$ is the Sylow $p$-subgroup of $Z$ and $\lambda_p\in\Irr(Z_p)$. Then $\phi_H(1)^2=|H:Z_{p'}|$ and $\lambda_{p'}^H=e\phi_H$ where $e=\phi(1)$ by \cite[Lemma~8.2]{Navarro2}. In particular, $\phi_H$ is the only character of $H$ lying over $\lambda_{p'}$. Let $\phi'\in\IBr(G)$ be a constituent of $\lambda^G$. Then $\phi'_H$ is a multiple of $\phi_H$. But since $H$ is a $p$-complement, $\phi'$ is uniquely determined by $\phi'_H$. It follows that $\phi'$ is a multiple of $\phi$, but then of course $\phi'=\phi$. Therefore, $\lambda$ is fully ramified in $G$.
\end{proof}

It has been remarked in \cite{NavarroFully} that \autoref{proppcentral} does not hold for the non-solvable group $G=\SL(2,17)$ when $p=17$. Another counterexample with a different flavor is the direct product $G=\SL(2,5)^2$ for $p=5$. Indeed, $\SL(2,5)$ has Brauer characters of degree $3$ and $4$. So $G$ has a Brauer character of degree $12$. 

Now we prove a strong version of \autoref{mainGagola}.

\begin{Thm}\label{ThmGa}
Let $G$ be a solvable group and $p$ be a prime. Then there exists a solvable $H$ with the following properties:
\begin{enumerate}[(i)]
\item $Z:=\Z(H)$ has square-free $p'$-order.
\item There exists a faithful Brauer character $\phi\in\IBr(H)$ with $\phi(1)^2=|H:Z|_{p'}$. In particular, $H$ has $p$-central type.
\item $G$ is isomorphic to a subgroup of $H/Z$.
\item $|G|$, $|H|$ and $|Z|$ have the same prime divisors apart from $p$.
\end{enumerate} 
\end{Thm}
\begin{proof}
We follow closely Gagola's construction~\cite[Theorem~1.2]{GagolaFully}.
Since the trivial group is of $p$-central type, we may assume that $G\ne 1$.
Let $M\unlhd G$ such that $q:=|G:M|$ is a prime. By induction on $|G|$, there exists a solvable group $K$ and a faithful $\mu\in\IBr(K)$ with $\mu(1)^2=|K:\Z(K)|_{p'}$ fulfilling the conclusion for $M$ instead of $G$. 
Let 
\[W:=\{(x_1,\ldots,x_q)\in\Z(K)^q:x_1\ldots x_q=1\}\le\Z(K^q)\]
and define $U:=K^q/W$. Let $z:=(x_1,\ldots,x_q)W\in\Z(U)$. For $g\in K$ and $1\le i\le q$ we have \[1=[(1,\ldots,1,g,1,\ldots,1)W,z]=(1,\ldots,1,[g,x_i],1,\ldots,1)W\]
and thus $[g,x_i]=1$. This shows that $\Z(U)=\Z(K)^q/W\cong\Z(K)$ and $U/\Z(U)\cong (K/Z(K))^q$. 
Since $\mu$ is faithful, we have $\pcore_p(K)=1=\pcore_p(U)$ by \cite[Lemma~2.32]{Navarro}. 

Let $\tau:=\mu\times\ldots\times\mu\in\IBr(K^q)$. Let $\lambda\in\IBr(\Z(K))$ be a constituent of $\mu_{\Z(K)}$. Then for $(x_1,\ldots,x_q)\in W$ we have $\tau(x_1,\ldots,x_q)=\tau(1)\lambda(x_1\ldots x_q)=\tau(1)$ (note that $\Z(K)$ and $W$ are $p'$-groups by induction). Hence, $W\le\Ker(\tau)$ by \cite[Lemma~6.11]{Navarro}. Conversely, let $x:=(x_1,\ldots,x_q)\in\Ker(\tau)$. Suppose that $x_i\notin\Z(K)$. Since $\mu$ is faithful, we obtain $|\mu(x_i)|<\mu(1)$ and $|\tau(x)|<\tau(1)$, a contradiction. Hence, $x\in\Z(K)^q$ and $\tau(1)\lambda(x_1\ldots x_q)=\tau(x)=\tau(1)$. Since $\mu$ and $\lambda$ are faithful, it follows that $x\in W$. Consequently, we may consider $\tau$ as a faithful Brauer character of $U$. 

\textbf{Case~1:} $q=p$.\\
Let $\alpha\in\Aut(K^p)$ be the shift automorphism such that $\alpha(x_1,\ldots,x_p):=(x_p,x_1,\ldots,x_{p-1})$. Clearly, $\alpha(W)=W$ and we may define $H:=U\rtimes\langle\alpha\rangle$. Let $z:=((x_1,\ldots,x_p)W,\alpha^i)\in\Z(H)$ with $0\le i<p$. For $g\in K\setminus\Z(K)$ we compute
\[1=[(g,1,\ldots,1)W,z]=(g,1,\ldots,1,x_{i+1}g^{-1}x_{i+1}^{-1},1,\ldots,1)W\]
and thus $i=0$ and $x_1,\ldots,x_p\in\Z(K)$. Conversely, it is easy to see that $\Z(U)\le\Z(H)$. Hence, $Z=\Z(H)=Z(U)$ and 
\[H/Z\cong (U/\Z(U))\rtimes\langle\alpha\rangle\cong (K/\Z(K))\wr C_p.\]
Moreover, $|G|$, $|H|$ and $|Z|$ have the same prime divisors apart from $p$.
By the universal embedding theorem (see \cite[Theorem~2.6A]{DM}), $G$ is isomorphic to a subgroup of 
\[M\wr(G/M)\le (K/\Z(K))\wr C_p\cong H/Z.\]  
Observe that $\tau$ is $H$-invariant. Hence, by Green's theorem $\tau$ has a (unique) extension $\phi\in\IBr(H)$ (see \cite[Theorem~8.11]{Navarro}). Then 
\[\phi(1)^2=\tau(1)^2=\mu(1)^{2p}=|K/\Z(K)|_{p'}^p=|H:Z|_{p'}\] 
and $H$ is of $p$-central type by \autoref{proppcentral}.

\textbf{Case~1:} $q\ne p$.\\
Here we need to modify the construction along the lines of \cite{GagolaFully}. Suppose first that $|\Z(K)|$ is not divisible by $q$. Let $Q:=\langle x\rangle\cong C_q$ and $\rho\in\Irr(Q)$ be faithful. We replace $(K,\mu)$ by $(K_1,\mu_1):=(K\times Q,\mu\times \rho)$. 
Then $\Z(K_1)=\Z(K)\times Q$ has square-free $p'$-order and $\mu_1(1)^2=|K_1:\Z(K_1)|_{p'}$. Since $\mu$ is faithful, $\pcore_p(K_1)=\pcore_p(K)=1$. Suppose that $(k,x^i)\in\Ker(\mu_1)$ with $k\in K$. Then $|\mu(k)|=\mu(1)$. Thus, $k\in\Z(K)$ and $\mu(k)$ is an integral multiple of a root of unity. On the other hand, 
\[\mu(k)=\mu(1)/\rho(x^i)\in\QQ_{|\Z(K)|}\cap\QQ_q=\QQ,\]
where $\QQ_n$ denotes the $n$-th cyclotomic field. 
This implies 
$\mu(k)=\pm\mu(1)$. In the case $\mu(k)=-\mu(1)$ we have $q\ne 2$ and therefore $\rho(x^i)=1$. This contradicts $\mu_1(k,x^i)=\mu_1(1)$. 
Hence, $\mu(k)=\mu(1)$, $\rho(x^i)=1$ and $(k,x^i)=1$. Thus, $\mu_1$ is faithful. From now on we may assume that $q$ divides $|\Z(K)|$.  

Let $z\in\Z(K)$ be an element of order $q$ and let $Y:=\langle y\rangle\cong C_q$. It is easy to check that the map
\[\alpha:U\times Y\to U\times Y,\qquad((x_1,\ldots,x_q)W,y^i)\mapsto((z^ix_q,x_1,x_2,\ldots,x_{q-1})W,y^i)\]
is a well-defined automorphism of order $q$. We define $H:=(U\times Y)\rtimes\langle\alpha\rangle$. 
Let 
\[h:=((x_1,\ldots,x_q)W,y^i,\alpha^j)\in\Z(H).\] 
Then $\alpha^j$ commutes with $y$ and therefore, $\alpha^j=1$. Moreover,
\[1=[\alpha,h]=(z^ix_qx_1^{-1},x_1x_2^{-1},\ldots,x_{q-1}x_q^{-1})W\]
implies $z^i=1=y^i$. This shows that $h\in\Z(U)$. Conversely, it is easy to see that $\Z(U)\le\Z(H)=Z$. 
Therefore, $Z=\Z(U)\cong\Z(K)$ has square-free $p'$-order and 
\begin{equation}\label{eqHZ}
|H:Z|=\frac{|U|q^2}{|\Z(K)|}=q^2|K:\Z(K)|^q.
\end{equation}
It follows that $|G|$, $|H|$ and $|Z|$ have the same prime divisors apart from $p$. A closer look shows that
\[H/Z\cong \bigl((K/\Z(K))^q\rtimes\langle\alpha\rangle\bigr)\times Y\cong (K/\Z(K)\wr C_q)\times C_q.\]
By the universal embedding theorem, $G$ is isomorphic to a subgroup of $M\wr(G/M)\le K/\Z(K)\wr C_q\le H/Z$. 

Since $\alpha(y)=((z,1,\ldots,1)W,y)$ and $\mu(z)\ne\mu(1)$, the Brauer character $\tau\times 1_Y\in\IBr(U\times Y)$ is not invariant under $\alpha$. By Clifford's theorem, $\phi:=(\tau\times 1_Y)^H\in\IBr(H)$ and $\Ker(\phi)\le Y\cap\alpha(Y)=1$. Moreover, 
\[\phi(1)^2=\tau(1)^2q^2=\mu(1)^{2q}q^2=|K:\Z(K)|_{p'}^qq^2\overset{\eqref{eqHZ}}{=}|H:Z|_{p'}.\qedhere\] 
\end{proof}

\section{Applications to blocks}\label{secappl}

For $N\unlhd G$ and $\lambda\in\IBr(N)$ we define $\IBr(G|\lambda)$ as the set of irreducible constituents of $\lambda^G$ as usual. Recall from \cite[Corollary~8.7]{Navarro} that $\phi\in\IBr(G|\lambda)$ if and only if $\lambda$ is a constituent of $\phi_N$. 

\begin{proof}[Proof of \autoref{mainplength}]
We may assume that $l\ge 2$. 
By \autoref{ThmGa}, there exists a solvable group $H$ of $p$-central type with $l_p(H)\ge l$ and $Z:=\Z(H)$ a $p'$-group.
Let $H/Z$ act faithfully on an elementary abelian $p$-group $V$ (for instance, the regular module) and define $G:=V\rtimes H$. Then $Z=\pcore_{p'}(G)$. Let $\lambda\in\IBr(Z)$ be fully ramified in $H$. Since $\IBr(G)=\IBr(H)$, $\lambda$ is also fully ramified in $G$. By a theorem of Fong (see \cite[Theorem~10.20]{Navarro}), there exists a block $B$ of $G$ of maximal defect with $\IBr(B)=\IBr(G|\lambda)$. In particular, $l(B)=1$. 
Let $(V,b)$ be a $B$-subpair, i.\,e. $b$ is a Brauer correspondent of $B$ in $\C_G(V)=V\times Z$. Then $\IBr(b)=\{1_V\times\lambda\}$ and $\N_G(V,b)=G$. Since $G/\C_G(V)\cong H/Z$ is not a $p$-group (as $l_p(H/Z)=l_p(H)\ge l\ge 2$), $B$ is not nilpotent.
\end{proof}

We give a concrete example starting with $H:=\mathtt{SmallGroup}(54,8)\cong 3^{1+2}_+\rtimes C_2$ of $2$-central type. This group acts on $V\cong C_2^4$ and
\[G:=V\rtimes H\cong\mathtt{SmallGroup}(864,3996)\]
has $2$-length $2$. Moreover, $G$ has a $2$-block $B$ of maximal defect and $l(B)=1$. We remark that $B$ covers the a non-principal block of $V\rtimes 3^{1+2}$, which was investigated in \cite{LS}.

\begin{proof}[Proof of \autoref{mainnonsolv}]
Let $b_0$ be the principal $p$-block of $Q$. Then $b\otimes b_0$ is a $G$-invariant block of $H\times Q$. 
By \cite[Corollary~9.21]{Navarro}, $B:=(b\otimes b_0)^G$ is the unique block covering $b\otimes b_0$. 
Since $A$ acts regularly on $\IBr(b\otimes b_0)$, Clifford theory implies that $l(B)=1$. Note that $(Q,b\otimes b_0)$ is a $B$-subpair. Since, $G/HQ\cong A$ is not a $p$-group (as $|A|=l(b)$ is not a $p$-power), $B$ cannot be nilpotent. 
\end{proof}

It is not so easy to find groups $H$ with blocks $b$ where $\Aut(H)$ acts transitively on $\IBr(b)$. 
The quasisimple groups $H$ with the desired property were investigated and partially classified in \cite{MNS}. For alternating and sporadic groups, $b$ must have defect $1$. Here, $l(b)$ divides $p-1$ and one can choose $A\cong C_{l(b)}$ and $Q\cong C_p$. The smallest case, $p=3$ and $H=\SL(2,5)$, leads to the group
\[G:=(H\times C_3)\rtimes C_2=H\rtimes S_3=\mathtt{SmallGroup}(720,414)\] 
with a non-nilpotent $3$-block $B$ with defect $2$ and $l(B)=1$. The same construction works with the simple group $H=\PSL(2,11)$.
Similarly examples can be obtained by wreath products like
\[G=(H\times H)\rtimes\langle(\alpha,1)\sigma\rangle\cong H^2\rtimes C_4\le \Aut(H^2)\]
where $\alpha\in\Aut(H)$ and $\sigma$ interchanges the two copies of $H$. By Kessar~\cite{KessarC3C3}, such blocks are Morita equivalent to blocks of solvable groups.
For $p=5$, one can take $H=\PSL(2,19)$ (there is certainly an infinite family). As long as we take blocks $b$ of $H$ with cyclic defect group and conjugate Brauer characters in $\Aut(H)$, the Brauer tree of $b$ is a star (otherwise there is no graph automorphism permuting the Brauer characters). Then $b$ is a so-called \emph{inertial} block, i.\,e. $b$ is basically Morita equivalent to its Brauer correspondent in the normalizer of a defect group (see \cite{Puignilex}). The same must be true for the block $b\otimes Q$ of $H\times Q$. Now a theorem of Zhou~\cite[Corollary]{ZhouInertial} implies that $B$ is inertial. In particular, $B$ is Morita equivalent to a block of a solvable group.

In general, the block of $H\rtimes A$ covering $b$ is often nilpotent. In this case, a theorem of Puig~\cite{Puignilex} shows that $b$ is inertial.

There are examples where $b$ has non-cyclic defect groups.
For $p=2$, we start with $H:=\PSU(3,5)$ where $b$ has defect group $C_2^2$ and $l(b)=3$ ($b$ is Morita equivalent to the principal block of $S_4$). We take $A\cong C_3$ and $Q=C_2^2$. Then $G:=H\rtimes A_4$ has a $2$-block with defect group $C_2^4$. By Eaton~\cite{EatonE16} this block is again Morita equivalent to a block of a solvable group.

We now construct a block $B$ with $\IBr(B)=\{\phi\}$ such that $\phi$ has a unique lift. If a group $H$ acts on a group $V$, we denote the stabilizer of $v\in V$ in $H$ by $H_v$. 

\begin{Thm}\label{unilift}
Let $H$ be a $p'$-group of central type. Suppose that $H/\Z(H)$ acts faithfully on an elementary abelian $p$-group $V$ such that $|H:H_v|>|H_v:\Z(H)|$ for all $v\in V\setminus\{1\}$. Then $G:=V\rtimes H$ has a $p$-block $B$ such that $\IBr(B)=\{\phi\}$ and $\phi$ has a unique lift to $\Irr(B)$. 
\end{Thm}
\begin{proof}
Let $Z:=\Z(H)$ and choose a fully ramified character $\lambda\in\Irr(Z)$. Let $\lambda^H=e\phi$ for some $\phi\in\Irr(H)$. By Fong's theorem, there exists a block $B$ of $G$ such that $\Irr(B)=\Irr(G|\lambda)$ and $\IBr(B)=\IBr(G|\lambda)$. We may consider the inflation of $\phi$ as an ordinary character and as a Brauer character of $B$. It suffices to show that every $\chi\in\Irr(B)\setminus\{\phi\}$ has degree $\chi(1)>\phi(1)=e$. There exists some non-trivial character $\mu\in\Irr(V)$ such that $\mu\times\lambda$ is a constituent of $\chi_{VZ}$. Since $H$ is a $p'$-group, the actions of $H$ on $V$ and on $\Irr(V)$ are permutation-isomorphic (see \cite[Corollary~2.12]{Navarro2}). Thus, by Clifford theory and by the hypothesis we have 
\[\chi(1)\ge|G:G_{\mu\times\lambda}|=|H:H_\mu|>|H_\mu:Z|.\] 
Multiplying the last inequality by $|H:H_\mu|$ further yields $|H:H_\mu|>\sqrt{|H:Z|}=e$. Hence, $\chi(1)>\phi(1)$ as desired.
\end{proof}

The property $|H:H_v|>|H_v:\Z(H)|$ in \autoref{unilift} is fulfilled whenever $H/\Z(H)$ acts semiregularly on $V$, i.\,e. $H_v=1$ for all $v\in V\setminus\{1\}$. In this case, $V\rtimes(H/\Z(H))$ is a Frobenius group with complement $H/\Z(H)$. In particular, the Sylow subgroups of $H/\Z(H)$ are cyclic or quaternion groups. But then $H/\Z(H)$ has trivial Schur multiplier and there is no group $H$ of central type. 
Hence, $|H/\Z(H)|$ is a product of at least four (possibly equal) primes (keeping in mind that $|H/\Z(H)|$ is a square). 
Suitable groups $H$ can be found with GAP~\cite{GAPnew}. For instance, $H=\mathtt{SmallGroup}(128,144)$ is of central type and acts on $V\cong C_5^4$ such that $V\rtimes (H/\Z(H))\cong\mathtt{PrimitiveGroup}(625,166)$ with the desired property.

The condition in \autoref{unilift} can be relaxed by taking the degrees of the characters in $H_\mu$ (where $\mu\in\Irr(V)$) into account. This leads to a smaller example where $H\cong\mathtt{SmallGroup}(128,138)$, $V\cong C_3^4$ and $V\rtimes (H/\Z(H))\cong\mathtt{PrimitiveGroup}(81,33)$. An example for $p=2$ is given by 
\[H/\Z(H)\cong\mathtt{IrredSolMatrixGroup}(18,2,6,163)\cong\mathtt{SmallGroup}(81,9)\]
acting on $V\cong C_2^{18}$. 

We use the opportunity to construct yet another unusual class of blocks. 
Let $Z\unlhd H$ and $\lambda\in\Irr(Z)$. A theorem of Gallagher asserts that $|\Irr(H|\lambda)|$ equals the number of so-called $\lambda$-\emph{good} conjugacy classes of $H/Z$ (see \cite[Section~5.5]{Navarro2}). If $\lambda$ is fully ramified, then $1$ is the unique $\lambda$-good conjugacy class. One may ask when the opposite situation occurs, where all conjugacy classes of $H/Z$ are good. 
This is holds for instance whenever $\lambda$ extends to $G$ (then $|\Irr(H|\lambda)|=k(H/Z)$ by Gallagher's theorem). 
But the converse is not true. 
For instance, $H:=\mathtt{SmallGroup}(128,731)$ has a subgroup $Z\le\Z(H)\cap H'$ of order $2$ such that $k(H)=2k(H/Z)$. The non-trivial character of $Z$ cannot extend to $H$ since $Z\le H'$.  
Now $H/Z$ acts non-trivially on $V\cong C_3^6$. It turns out that the solvable group $G:=V\rtimes H$ has two $3$-blocks $B_0$ and $B_1$ with defect group $V$ and inertial quotient $H/Z$ such that $k(B_0)=k(B_1)=84$ and $l(B_0)=l(B_1)=16$ ($B_0$ is the principal block of $G$). However, $B_0$ and $B_1$ are not Morita equivalent, because they have distinct decomposition matrices (the former matrix contains a $4$, but the latter does not).

\section*{Acknowledgment}
The work on this paper started during a research stay at the University of Valencia in February 2023. I thank Gabriel Navarro and Alexander Moretó for the great hospitality received there. \autoref{mainNST} was initiated by a question of Britta Späth at the Oberwolfach workshop “Representations of finite groups” (ID 2316) in April 2023. I further thank Radha Kessar for some helpful discussions on this paper. 
This paper is supported by the German Research Foundation (\mbox{SA 2864/4-1}).

\end{document}